\documentclass[11pt,english]{article}
\usepackage[T1]{fontenc}
\usepackage[latin9]{inputenc}
\usepackage{amsmath}
\usepackage{amssymb}
\usepackage{graphicx}
\usepackage{wasysym}

\makeatletter
\usepackage{amsfonts}
\usepackage{amsmath}
\usepackage{amsthm}
\usepackage[margin=1.5in]{geometry}
\usepackage{amsfonts}\usepackage{amsthm}\usepackage{color}\usepackage{enumitem}\usepackage{bm}\usepackage{cancel}\usepackage{thmtools}

\date{}

\def\theenumi{\arabic{enumi}}

\def\theenumii{\alph{enumii}}
\def\p@enumii{\theenumi.}

\def\theenumiii{\arabic{enumiii}}
\def\p@enumiii{(\theenumi)(\theenumii)}

\def\p@enumiv{\p@enumiii.\theenumiii}
\newtheorem{theorem}{Theorem}[section]

\newtheorem{corollary}[theorem]{Corollary}

\newtheorem{lemma}[theorem]{Lemma}

\newtheorem{proposition}[theorem]{Proposition}

\theoremstyle{definition}
\newtheorem{acknowledgement*}[theorem]{Acknowledgement}
\newtheorem{definition}[theorem]{Definition}

\newtheorem{remark}[theorem]{Remark}
\newtheorem{example}[theorem]{Example}

\newcommand{\id}{\mathrm{id}}

\makeatother

\usepackage{babel}
\begin{document}

\title{On groups, slow heat kernel decay yields Liouville property and sharp
entropy bounds}

\author{\textsc{Yuval Peres and Tianyi Zheng}}
\maketitle
\begin{abstract}
Let $\mu$ be a symmetric probability measure of finite entropy on
a group $G$. We show that if $-\log\mu^{(2n)}(\id)=o(n^{1/2})$,
then the pair $(G,\mu)$ has the Liouville property (all bounded $\mu$-harmonic
functions on $G$ are constant). Furthermore, if $-\log\mu^{(2n)}(\id)=O(n^{\beta})$
where $\beta\in(0,1/2)$, then the entropy of the $n$-fold convolution
power $\mu^{(n)}$ satisfies $H(\mu^{(n)})=O\left(n^{\frac{\beta}{1-\beta}}\right)$.
These results improve earlier work of Gournay \cite{Gournay2014},
Saloff-Coste and the second author \cite{Saloff-Coste2014}. We illustrate
the sharpness of the bounds on a family of groups.
\end{abstract}

\section{Introduction}

Let $G$ be a finitely generated infinite group equipped with a generating
set $S$, and let $\mu$ be a probability measure on $G$. Denote
by $\left|g\right|$ the word distance of $g$ from the identity on
the Cayley graph induced by $S$. Let $X_{1},X_{2},\ldots$ be a sequence
of i.i.d.\ random variables with distribution $\mu$, so $W_{n}=X_{1}\cdots X$
is the random walk on $G$ with step distribution $\mu$. The law
of $W_{n}$ is the $n$-fold convolution power $\mu^{(n)}$. The return
probability of the $\mu$-random walk to the identity $\id$ after
$2n$ steps is 
\[
\mathbf{P}\left(W_{2n}=\mathrm{id}\right)=\mu^{(2n)}(\id).
\]
 The Shannon entropy of $W_{n}$ is 
\[
H_{\mu}(n)=H(W_{n})=-\sum_{x\in G}\mu^{(n)}(x)\log\mu^{(n)}(x).
\]
The rate of escape\emph{ }of $W_{n}$ is 
\[
L_{\mu}(n)=\mathbf{E}\left|X_{n}\right|=\sum_{x\in G}|x|\mu^{(n)}(x).
\]
The pair $\left(G,\mu\right)$ has the \emph{Liouville property} if
all bounded $\mu$-harmonic functions on $G$ are constant. By classical
work of Avez \cite{Avez1976}, Derrienic \cite{Derriennic1980} and
Kaimanovich-Vershik \cite{KV}, for $\mu$ with finite entropy $H_{\mu}(1)<\infty$,
the pair $(G,\mu)$ has the Liouville property if and only if the
Avez \emph{asymptotic entropy} $h_{\mu}=\lim_{n\to\infty}\frac{H_{\mu}(n)}{n}$
is $0$. We say a probability measure $\mu$ on $G$ is \emph{symmetric}
if $\mu(g)=\mu(g^{-1})$ for all $g\in G$. 

Our goal in this article is to show a link between the decay of the
return probability and the growth of entropy of a symmetric random
walk on $G$. Namely, we derive an upper bound on $H_{\mu}(n)$ from
a lower bound on $\mu^{(2n)}(\id)$, provided that $\mu^{(2n)}(\id)$
decays sufficiently slowly. 

\begin{theorem}\label{liouville}

Suppose $\mu$ is a symmetric probability measure of finite entropy
on $G$ such that 
\[
\mu^{(2n)}(\id)\ge\exp\left(-\gamma(n)\right)
\]
where $\gamma:[1,\infty)\to\mathbb{R}_{+}$ is an unbounded function
such that both $\gamma(n)$ and $n^{\frac{1}{2}}/\gamma(n)$ are increasing,
that satisfies 
\[
\lim_{n\to\infty}\frac{\gamma(n)}{n^{\frac{1}{2}}}=0.
\]
 Then $(G,\mu)$ has the Liouville property. 

\end{theorem}

The decay of the return probability enjoys good stability properties,
in particular the asymptotic decay of return probability of simple
random walk is a quasi-isometry invariant, see Pittet and Saloff-Coste
\cite{PSCstab}. However, it remains a major open problem whether
the Liouville property is stable under changing the generating set
of the group. We deduce from Theorem \ref{liouville} the following
corollary regarding stability of the Liouville property provided that
$\mu^{(2n)}(\id)$ decays slower than $\exp\left(-n^{1/2}\right)$.
We say a probability measure $\mu$ on $G$ has finite second moment
if $\sum_{g\in G}|g|^{2}\mu(g)<\infty$. 

\begin{corollary}\label{stability}

Suppose $G$ is a finitely generated group such that for some symmetric
probability measure $\mu$ with finite generating support on $G$,
$\mu^{(2n)}(\id)$ satisfies the assumptions in Theorem \ref{liouville}.
Let $\Gamma$ be a finitely generated group that is quasi-isometric
to $G$. Then $(\Gamma,\eta)$ has the Liouville property for any
symmetric probability measure $\eta$ of finite second moment on $\Gamma$.

\end{corollary}

Simple random walk on the lamplighter group over the two-dimensional
lattice $G=\mathbb{Z}_{2}\wr\mathbb{Z}^{2}$ satisfies\footnote{Given two monotone functions $\phi,\psi$, write $\phi\apprge\psi$
if there are constants $c_{1},c_{2}\in(0,\infty)$, such that $\phi(t)\ge c_{1}\psi(c_{2}t)$
(using integer values if $\phi,\psi$ are defined on $\mathbb{N}$).
We write $\psi\simeq\phi$ if both $\phi\apprge\psi$ and $\psi\apprge\phi$
holds.} $\mu^{(2n)}(\id)\simeq\exp(-n^{1/2})$ and $H_{\mu}(n)\simeq n/\log n$,
see \cite{Erschlerdrift,Erschler2006,Pittet2002}. This example is
just beyond the limit of application of Theorem \ref{liouville}.
Kotowski and Virág \cite{Kotowski2015} analyzed a group $G$ on which
simple random walk satisfies $\mu^{(2n)}(\id)\gtrsim\exp(-n^{1/2+o(1)})$
and the entropy $H_{\mu}(n)$ has linear growth. The Kotowski-Virág
example shows that the exponent $1/2$ is the critical value in the
setting of Theorem \ref{liouville}. It is an interesting open problem
whether $\mu^{(2n)}(\id)\gtrsim\exp\left(-n^{\frac{1}{2}}\right)$
implies that $(G,\mu)$ has the Liouville property. 

By a result of Kaimanovich \cite{Kaimanovich1991}, if $\mu$ is a
symmetric probability measure with finite first moment on a polycyclic
group $G$, then $(G,\mu)$ has the Liouville property. By Alexopoulos
\cite{Alex}, Coulhon-Grigor'yan-Pittet \cite[Theorem 7.10]{CGP},
for simple random walk on a polycyclic group of exponential volume
growth, the return probability satisfies $\mu^{(2n)}(\id)\simeq\exp\left(-n^{1/3}\right)$.
As a corollary of Theorem \ref{liouville}, we extend the aforementioned
result of Kaimanovich to groups where the return probability of simple
random walk decays no faster than $\exp\left(-n^{1/3}\right)$. 

\begin{corollary}\label{first-m}

Suppose $G$ is a finitely generated group such that for some symmetric
probability measure $\mu$ with finite generating support on $G$,
\[
\mu^{(2n)}(\id)\gtrsim\exp\left(-n^{1/3}\right).
\]
Then $(G,\eta)$ has the Liouville property for every symmetric probability
measure $\eta$ on $G$ with finite first moment $\sum_{g\in G}|g|\eta(g)<\infty$.

\end{corollary}

So far all known examples of groups that satisfy the assumption of
Corollary \ref{first-m} belong to the class of \emph{geometrically
elementary solvable groups }(GES), defined in Tessera \cite{Tessera2013}.
See Section \ref{sec:escape} for the more details.

When the decay of the return probability $\mu^{(2n)}(\id)$ is much
slower than $\exp\left(-n^{\frac{1}{2}}\right)$, we have the following
explicit entropy upper bound. It improves a bound from \cite{Saloff-Coste2014}.

\begin{theorem}\label{bound-1}

Let $\mu$ be a symmetric probability measure of finite entropy on
$G$. Suppose there exists constants $C>0$, $\beta\in\left(0,\frac{1}{2}\right)$
such that 
\[
\mu^{(2n)}(\id)\ge\exp\left(-Cn^{\beta}\right).
\]
Then there exists a constant $C_{1}=C_{1}(\beta,C)$ such that 
\[
H_{\mu}(n)\le C_{1}n^{\frac{\beta}{1-\beta}}.
\]

\end{theorem}

\begin{remark}

The bound stated above is a special case of Theorem \ref{bound}.
A larger class of lower bound functions on the return probability
can be converted into entropy upper bounds. For example, suppose there
exists constants $C>0$, $\beta\in\left(0,\frac{1}{2}\right)$ and
$\kappa\in\mathbb{R}$ such that 
\[
\mu^{(2n)}(\id)\ge\exp\left(-Cn^{\beta}\log^{\kappa}(n+1)\right),
\]
then there exists a constant $C_{1}=C_{1}(\beta,\kappa,C)$ such that
\[
H_{\mu}(n)\le C_{1}n^{\frac{\beta}{1-\beta}}\log^{\frac{\kappa}{1-\beta}}(n+1).
\]
More details can be found in Section \ref{sec:entropy}.

\end{remark}

Based on the spectral profile of balls in the group $G$, we derive
an upper bound on the rate of escape $L_{\mu}(n)$. Given a symmetric
probability measure $\mu$ on $G$ and a function $f\in\ell^{2}(G)$,
consider the associated Dirichlet form 
\[
\mathcal{E}_{\mu}(f)=\frac{1}{2}\sum_{x,y\in G}(f(xy)-f(x))^{2}\mu(y),
\]
and define 
\begin{equation}
\lambda_{\mu}(\Omega)=\inf\{\mathcal{E}_{\mu}(f):\mbox{support}(f)\subset\Omega,\|f\|_{2}=1\}.\label{def-eig}
\end{equation}
In words, $\lambda_{\mu}(\Omega)$ is the lowest eigenvalue of the
operator of convolution by $\delta_{e}-\mu$ with Dirichlet boundary
condition in $\Omega$. This operator is associated with the discrete
time Markov process corresponding to the $\mu$-random walk killed
outside $\Omega$.

\begin{theorem}\label{escape}

Let $\mu$ be a symmetric probability measure on $G$, $\left(W_{n}\right)$
be a $\mu$-random walk on $G$. Suppose there exists constants $C>0$
and $\theta>0$ such that 
\[
\lambda_{\mu}(B(\id,r))\le Cr^{-\theta},
\]
then for any $\alpha\in\left(0,\theta\right)$, there exists a constant
$C_{1}=C_{1}(C,\theta,\alpha)$ such that 
\[
\mathbf{E}\max_{0\le k\le n}\left|W_{k}\right|^{\alpha}\le C_{1}n^{\alpha/\theta}.
\]

\end{theorem}

The stability of the asymptotic behavior of the function $r\to\lambda_{\mu}(B(\id,r))$
was investigated by Tessera in \cite{Tessera2013}. In particular,
in \cite{Tessera2013} it was proved that an upper bound of the form
$\lambda_{\mu}(B(\id,r))=O(g(Cr))$ when $r\to\infty$ for some monotone
function $g$ is stable under quasi-isometry, taking finite products,
quotients and finitely generated subgroups. From the stability results
we deduce the following corollary.

\begin{corollary}\label{polycyclic}

Let $G$ be a finitely generated group that is quasi-isometric to
a polycyclic group. Then for any symmetric probability measure $\mu$
on $G$ with finite second moment, there exists a constant $C>0$
such that
\[
L_{\mu}(n)\le\mathbf{E}\max_{0\le k\le n}\left|W_{k}\right|\le Cn^{\frac{1}{2}}.
\]

\end{corollary}

Theorem \ref{escape} is a special case of Theorem \ref{moment},
and Corollary \ref{polycyclic} is a special case of Corollary \ref{diffusive}.
In Section \ref{sec:bubble} we illustrate the sharpness of Theorem
\ref{bound-1}.

\begin{proposition}\label{bubble-exponent}

For any $\beta\in\left[\frac{1}{3},\frac{1}{2}\right)$, there exists
a group $G$ and a symmetric probability measure $\mu$ of generating
finite support on $G$ such that 
\begin{align*}
\mu^{(2n)}(\id) & \simeq\exp\left(-n^{\beta}\log^{1-\beta}n\right),\\
H_{\mu}(n) & \simeq n^{\frac{\beta}{1-\beta}}\log n.
\end{align*}

\end{proposition}

Theorem \ref{liouville}, \ref{bound-1} and \ref{escape} extend
to symmetric random walks on transitive graphs. In Section 6, we explain
how the results carry over. 

\subsubsection{Related works}

The idea of connecting slow decay of the return probability to the
Liouville property was first proposed by Gournay in \cite{Gournay2014}.
The \emph{Hilbert compression exponent }$\alpha_{2}^{\ast}(G)$ was
introduced in Guentner and Kaminker \cite{Guentner2004}, who defined
$\alpha_{2}^{\ast}(G)$ as the supremum over all $\alpha\ge0$ such
that there exists a $1$-Lipschitz map $f:G\to\mathcal{H}$ and a
constant $c>0$ such that $\left\Vert f(x)-f(y)\right\Vert _{\mathcal{H}}\ge cd_{G}(x,y)^{\alpha}$.
Analogously, one can consider the \emph{equivariant Hilbert compression
exponent $\alpha_{2}^{\#}(G)$, }which is defined in the same way
as $\alpha_{2}^{*}(G)$ with the additional requirement that $f$
is equivariant (see Section 3 for the definition). It is known that
for finitely generated amenable groups, $\alpha_{2}^{\ast}(G)=\alpha_{2}^{\#}(G)$,
see \cite{Cornulier2007}. One result from \cite{Gournay2014} is
the following. Suppose $\mu$ is a symmetric probability measure of
finite support on $G$ such that $\mu^{(2n)}(\id)\ge C_{1}\exp\left(-C_{2}n^{\gamma}\right)$
for some $\gamma\in(0,1)$, then \cite[Theorem 1.1]{Gournay2014}
\[
\alpha_{2}^{\#}(G)\ge\frac{1-\gamma}{1+\gamma}.
\]
Let $\beta^{\ast}(\mu)$ be the upper speed exponent, $\beta^{*}(\mu):=\limsup_{n\to\infty}\frac{\log L_{\mu}(n)}{\log n}$.
From Austin, Naor and Peres \cite[Proposition 1.1]{Austin2009}, 
\begin{equation}
\alpha_{2}^{\#}(G)\le\frac{1}{2\beta^{\ast}(\mu)}.\label{eq:anp}
\end{equation}
Therefore Gournay's result implies that when $\mu^{(2n)}(\id)\ge C_{1}\exp\left(-C_{2}n^{\gamma}\right)$,
we have 
\begin{equation}
\beta^{\ast}(\mu)\le\frac{1+\gamma}{2(1-\gamma)}.\label{eq:beta}
\end{equation}
By a result of Varopoulos, we have $\mu^{(2n)}(id)\lesssim\exp\left(-n^{\frac{1}{3}}\right)$
for simple random walks on groups of exponential volume growth. The
upper bound (\ref{eq:beta}) on $\beta^{\ast}(\mu)$ is not strong
enough to imply the Liouville property of $(G,\mu)$ when $\gamma\in[1/3,1/2)$. 

In the work of the second author with Saloff-Coste \cite{Saloff-Coste2014},
an upper bound similar to Theorem \ref{bound-1} was proved. More
precisely, suppose $\mu$ is a symmetric probability measure of finite
second moment on $G$ such that $\mu^{(2n)}(\id)\ge C_{1}\exp\left(-C_{2}n^{\gamma}\right)$
for some $\gamma\in(0,1/2)$, then by \cite[Theorem 1.7]{Saloff-Coste2014}
\[
H_{\mu}(n)\lesssim\left(n\log^{1+\epsilon}n\right)^{\frac{\gamma}{1-\gamma}}\ \mbox{for any }\epsilon>0.
\]
In particular, this bound implies the Liouville property of $(G,\mu)$
for $\gamma\in(0,1/2)$. Our results in this paper imply a sharper
upper bound $H_{\mu}(n)\lesssim n^{\frac{\gamma}{1-\gamma}}$ under
the same assumption, and are applicable to a larger class of decay
lower bounds. 

The following bound on the upper speed exponent is known through work
of Tessera \cite{Tessera2011}, Austin, Naor and Peres \cite{Austin2009}.
Suppose there exists constants $C>0$ and $\theta>0$ such that 
\[
\lambda_{\mu}(B(\id,r))\le Cr^{-\theta},
\]
then 
\begin{equation}
\beta^{*}(\mu):=\limsup_{n\to\infty}\frac{\log L_{\mu}(n)}{\log n}\le\frac{1}{\theta}.\label{eq:upper-speed}
\end{equation}
This relation is proved via the equivariant Hilbert compression exponent
$\alpha_{2}^{\#}(G)$. Indeed, by \cite[Theorem 10]{Tessera2011},
the assumption $\lambda_{\mu}(B(\id,r))\le Cr^{-\theta}$ implies
that 
\begin{equation}
\alpha_{2}^{\#}(G)\ge\frac{\theta}{2}.\label{eq:tessera}
\end{equation}
By \cite[Proposition 1.1]{Austin2009}, we have 
\[
\beta^{\ast}(\mu)\le\frac{1}{2\alpha_{2}^{\#}(G)}\le\frac{1}{\theta}.
\]
Theorem \ref{escape} provides a more precise upper bound on $L_{\mu}(n)$
than just the upper speed exponent.

\section{The spectral profile}

The spectral profile $\Lambda_{G,\mu}$ of a symmetric probability
measure $\mu$ on $G$ is defined as 
\[
\Lambda_{G,\mu}(v)=\Lambda_{\mu}(v)=\inf\{\lambda_{\mu}(\Omega):\Omega\subset G,\;|\Omega|\le v\},
\]
where $\lambda_{\mu}(\Omega)$ is the lowest Dirichlet eigenvalue
defined in (\ref{def-eig}). We now review work of Coulhon \cite{CNash}
which relates the behavior of $n\mapsto\mu^{(2n)}(\id)$ to the spectral
profile $v\mapsto\Lambda_{\mu}(v)$. See \cite{CNash} for references
to earlier related works, in particular, work of Grigor'yan in which
the spectral profile plays a key role. By \cite[Proposition II.1]{CNash},
we have 
\[
\mu^{(2n+2)}(\id)\le2\psi(2n),
\]
where $\psi:[0,+\infty)\to[1,+\infty)$ is defined implicitly by 
\begin{equation}
t=\int_{1}^{1/\psi(t)}\frac{ds}{2s\Lambda_{\mu}(4s)}.\label{eq:coulhon}
\end{equation}
From Coulhon's result one can deduce the following useful lemma, see
\cite[Lemma 2.5]{Saloff-Coste2014}. We include a proof here for the
reader's convenience. 

\begin{lemma}\label{Lambda-upper}

Assume that $\mu^{(2n)}(\id)\ge\exp(-\gamma(n))$ where $\gamma:[1,\infty)\to(0,\infty)$
is an increasing function and $\lim_{n\to\infty}\gamma(n)=\infty$.
Then for all $n\in\mathbb{N}$ we have 
\[
\Lambda_{\mu}\left(8e^{\gamma(n)}\right)\le\frac{\gamma(n)+\log8}{2n}.
\]

\end{lemma}

\begin{proof}

Let $\psi$ be defined in terms of $\Lambda_{\mu}$ as in (\ref{eq:coulhon}).
By definition of $\psi$ and the fact that $\Lambda_{\mu}$ is a non-increasing
function, we have 
\[
t=\int_{1}^{1/\psi(t)}\frac{ds}{2s\Lambda_{\mu}(4s)}\le\frac{\log\big(4/\psi(t)\big)}{2\Lambda_{\mu}\big(4/\psi(t)\big)},
\]
which we rewrite as 
\[
\Lambda_{\mu}\big(4/\psi(t)\big)\le\frac{\log\big(4/\psi(t)\big)}{2t}.
\]
By \cite[Proposition II.1]{CNash} and the hypothesis, 
\[
e^{-\gamma(n+1)}\le\mu^{(2n+2)}(\id)\le2\psi(2n).
\]
Hence 
\[
\Lambda_{\mu}\left(8e^{\gamma(n+1)}\right)\le\Lambda_{\mu}\big(4/\psi(2n)\big)\le\frac{\log\big(4/\psi(2n)\big)}{4n}\le\frac{\gamma(n+1)+\log8}{4n}.
\]

\end{proof}

\section{Upper bounds on entropy\label{sec:entropy}}

In this section we prove Theorem \ref{liouville}, \ref{bound-1}
and Corollary \ref{stability}. We will use $1$-coboundaries constructed
from test functions provided by the spectral profile, together with
the Markov type inequality to bound the probability that the random
walk deviates from its typical behavior. Throughout this section,
let $\left(W_{n}\right)$ be a $\mu$-random walk on $G$, where $\mu$
is a symmetric probability measure of finite entropy $H_{\mu}(1)<\infty$. 

Recall that $f:G\to\mathcal{H}$, where $\mathcal{H}$ is a Hilbert
space, is called a $1$-cocycle if there exists a representation $\pi:G\to\mbox{Isom}(\mathcal{H})$
such that $f(gx)=\pi_{g}f(x)+f(g)$. The space of $1$-cocycles associated
with representation $\pi$ is denoted by $Z^{1}(G,\pi)$. The space
of $1$-coboundaries, $B^{1}(G,\pi)=\left\{ f:G\to\mathcal{H}:\ \exists v\in\mathcal{H},\ f(g)=\pi_{g}v-v\right\} $,
is a subspace of $Z^{1}(G,\pi)$. We say a map $f:G\to\mathcal{H}$
is equivariant if $f\in Z^{1}(G,\pi)$ for some representation $\pi:G\to\mbox{Isom}(\mathcal{H})$.

The following Markov type inequality is a special case of \cite[Theorem 2.1]{Naor2008},
see also \cite[Remark 2.6]{Naor2008}: for any $f\in Z^{1}(G,\pi)$,
we have for all $n\in\mathbb{N}$, 
\begin{equation}
\mathbf{E}\left[\|f(W_{n})\|_{\mathcal{H}}^{2}\right]\le n\mathbf{E}\left[\|f(W_{1})\|_{\mathcal{H}}^{2}\right].\label{eq:markov}
\end{equation}
Note that if we choose $f$ to be a 1-coboundary, $f(g)=\pi_{g}v-v$,
then $\|f(g)\|_{\mathcal{H}}\le2\left\Vert v\right\Vert _{\mathcal{H}}$
is bounded. In particular, if we apply (\ref{eq:markov}) to $1$-coboundaries,
quantities on both sides are finite. The following lemma holds for
all symmetric random walks on $G$. 

\begin{lemma}\label{tail}

Let $\mu$ be a symmetric probability measure on $G$, and $U$ be
a finite subset of $G$. Then for a $\mu$-random walk $\left(W_{n}\right)$
on $G$, 
\[
\mathbf{P}\left(W_{n}\notin U^{-1}U\right)\le n\lambda_{\mu}(U).
\]

\end{lemma}

\begin{proof}

By the definition of $\lambda_{\mu}(U)$, for any $\epsilon>0$, there
exists a function $\phi$ supported on $U$ such that 
\[
\frac{\mathcal{E}_{\mu}(\phi)}{\|\phi\|_{2}^{2}}\le\lambda_{\mu}\left(U\right)+\epsilon.
\]
Consider the $1$-coboundary 
\begin{align*}
\sigma:G & \to\ell^{2}(G)\\
\sigma(g) & =\tau_{g}\phi-\phi,
\end{align*}
where $\tau_{g}$ is the right translation $\tau_{g}\phi(x)=\phi(xg)$.
By the Markov type inequality in Hilbert space (\ref{eq:markov}),
we have 
\[
\mathbf{E}\left[\|\sigma(W_{n})\|_{2}^{2}\right]\le n\mathbf{E}\left[\|\sigma(W_{1})\|_{2}^{2}\right].
\]
Note that 
\[
\mathbf{E}\left[\|\sigma(W_{1})\|_{2}^{2}\right]=2\mathcal{E}_{\mu}\left(\phi\right).
\]
On the other hand, if $g\notin U^{-1}U$, then $\mbox{supp }\phi\cap\mbox{supp }\tau_{g}\phi=\emptyset$,
and it follows that 
\[
\|\sigma(g)\|_{2}^{2}=\|\phi\|_{2}^{2}+\|\tau_{g}\phi\|_{2}^{2}=2\|\phi\|_{2}^{2}.
\]
Therefore 
\[
2\|\phi\|_{2}^{2}\mathbf{P}\left(W_{n}\notin U^{-1}U\right)\le\mathbf{E}\left[\|\sigma(W_{n})\|_{2}^{2}\right]\le2n\mathcal{E}_{\mu}\left(\phi\right).
\]
From the choice of function $\phi$, we have 
\begin{equation}
\mathbf{P}\left(W_{n}\notin U^{-1}U\right)\le n\frac{\mathcal{E}_{\mu}\left(\phi\right)}{\|\phi\|_{2}^{2}}\le n\left(\lambda_{\mu}(U)+\epsilon\right).\label{eq:tail}
\end{equation}
Since $\epsilon>0$ is an arbitrary positive number, we obtain the
statement.

\end{proof}

To proceed, for each $k\in\mathbb{N}$, let $U_{k}$ be a set of size
$\left|U_{k}\right|=\left\lceil e^{e^{k}}\right\rceil $ such that
\[
\lambda_{\mu}\left(U_{k}\right)\le2\Lambda_{\mu}\left(e^{e^{k}}\right).
\]
By Lemma \ref{Lambda-upper}, the assumption $\mu^{(2n)}(\id)\ge\exp(-\gamma(n))$
implies 
\begin{equation}
\Lambda_{\mu}\left(e^{e^{k}}\right)\le\frac{e^{k}}{2\gamma^{-1}\left(e^{k}-\log8\right)}\le\frac{1}{2e^{k}\varphi(e^{k}/2)}\mbox{ for all }k\ge k_{0}=\log\left(\gamma(2)+\log64\right),\label{eq:L-psi}
\end{equation}
where $\varphi(x)=\frac{\gamma^{-1}(x)}{x^{2}}$. Note that if $x^{-\frac{1}{2}}\gamma(x)$
decreases to $0$ as $x\to\infty,$ then $\varphi(x)$ is increasing
and $\lim_{x\to\infty}\varphi(x)=\infty$.

\begin{proof}[Proof of Theorem \ref{liouville}]

For a given time $n$, define 
\[
\ell_{n}=\min\{k:\ e^{k}\varphi^{\frac{1}{2}}(e^{k}/2)\ge n\}.
\]
For example, when $\gamma(n)=Cn^{\beta}$ for some $\beta\in(0,1/2)$,
then $\varphi(x)=C^{-1/\beta}x^{1/\beta-2}$ and $\ell_{n}\sim\log\left(n^{\frac{\beta}{1-\beta}}\right)$. 

By Lemma \ref{tail}, we have 
\begin{equation}
\mathbf{P}\left(W_{n}\notin U_{\ell_{n}}^{-1}U_{\ell_{n}}\right)\le2n\Lambda_{\mu}\left(e^{e^{\ell_{n}}}\right)\le\frac{2n}{2e^{\ell_{n}}\varphi(e^{\ell_{n}}/2)}\le\frac{1}{\varphi^{\frac{1}{2}}\left(e^{\ell_{n}}\right)}.\label{eq:tail2}
\end{equation}
This implies that 
\begin{equation}
\lim_{n\to\infty}\mathbf{P}\left(W_{n}\notin U_{\ell_{n}}^{-1}U_{\ell_{n}}\right)=0.\label{eq:limit}
\end{equation}
Assume now that the asymptotic entropy $h_{\mu}>0$. Then by Shannon
theorem, see for example \cite[Theorem 14.10]{LPbook}, for any $\epsilon>0$,
there exists $n_{\epsilon}$ such that for all $n\ge n_{\epsilon},$
\[
\mathbf{P}\left(W_{n}\in U\right)\le\epsilon+|U|e^{-nh_{\mu}/2}.
\]
Note that since $e^{\ell_{n}-1}\varphi^{\frac{1}{2}}(e^{\ell_{n}-1}/2)<n$,
we have 
\[
\log\left|U_{\ell_{n}}^{-1}U_{\ell_{n}}\right|\le2\log\left|U_{\ell_{n}}\right|=2e^{\ell_{n}}<\frac{2en}{\varphi^{\frac{1}{2}}\left(e^{\ell_{n}-1}/2\right)}.
\]
 Therefore if $h_{\mu}>0$, we have 
\[
\lim_{n\to\infty}\mathbf{P}\left(W_{n}\in U_{\ell_{n}}^{-1}U_{\ell_{n}}\right)=0,
\]
 which contradicts (\ref{eq:limit}). We conclude that $h_{\mu}=0$. 

\end{proof}

We now derive Corollary \ref{stability} and \ref{first-m} from Theorem
\ref{liouville}.

\begin{proof}[Proof of Corollary \ref{stability}]

By work of Pittet and Saloff-Coste \cite[Theorem 1.4]{PSCstab}, if
$\mu_{i}$, $i=1,2$, are symmetric probability measures on $G$ with
generating support and finite second moment $\sum_{g\in G}\left|g\right|^{2}\mu_{i}(g)<\infty$,
then the functions $n\mapsto\mu_{i}^{(2n)}(\id)$ satisfy $\mu_{1}^{(2n)}(\id)\simeq\mu_{2}^{(2n)}(\id)$.
Further, the equivalence class of the decay function $\mu_{1}^{(2n)}(\id)$
is an invariant of quasi-isometry, see \cite{PSCstab}. Combine this
stability result with Theorem \ref{liouville}, we obtain the statement.

\end{proof}

\begin{proof}[Proof of Corollary \ref{first-m}]

From the proof of Theorem \ref{liouville}, it is sufficient to show
that under the assumptions of the statement, for any symmetric probability
measure $\eta$ on $G$ with finite first moment, we have 
\[
\Lambda_{\eta}(e^{x})\le\frac{1}{x\varphi(x)},\mbox{ where }\varphi(x)\to\infty\mbox{ as }x\to\infty.
\]
To obtain such an estimate, we use a truncation argument similar to
the proof of \cite[Theorem 2.13]{Saloff-Coste2014}. 

We may assume that $\mu$ is uniform on a symmetric finite generating
set $S$ of $G$. By Lemma \ref{Lambda-upper}, the assumption that
$\mu^{(2n)}(\id)\gtrsim\exp(-n^{1/3})$ implies that there is a constant
$C>0$ such that for all $r\ge1$, 
\[
\Lambda_{\mu}(e^{r})\le\frac{C}{r^{2}}.
\]
Recall the pseudo-Poincaré inequality (see for example \cite{PSCstab})
\begin{equation}
\left\Vert f-\tau_{g}f\right\Vert _{\ell^{2}(G)}^{2}\le2|S||g|^{2}\mathcal{E}_{\mu}(f),\label{eq:pp}
\end{equation}
where $\tau_{g}$ is the right translation $\tau_{g}f(x)=f(xg)$.
For any function $f$ with finite support, we have 
\begin{align*}
\mathcal{E}_{\eta}(f) & =\frac{1}{2}\sum_{|g|\le r}\left\Vert f-\tau_{g}f\right\Vert _{\ell^{2}(G)}^{2}\eta(g)+\frac{1}{2}\sum_{|g|>r}\left\Vert f-\tau_{g}f\right\Vert _{\ell^{2}(G)}^{2}\eta(g)\\
 & \le\frac{1}{2}\sum_{|g|\le r}\left\Vert f-\tau_{g}f\right\Vert _{\ell^{2}(G)}^{2}\eta(g)+\left\Vert f\right\Vert _{\ell^{2}(G)}^{2}\eta\left\{ g:|g|\ge r\right\} .
\end{align*}
By the pseudo-Poincaré inequality (\ref{eq:pp}), we have 
\[
\sum_{|g|\le r}\left\Vert f-\tau_{g}f\right\Vert _{\ell^{2}(G)}^{2}\eta(g)\le2|S|\mathcal{E}_{\mu}(f)\sum_{|g|\le r}|g|^{2}\eta(g).
\]
Therefore 
\[
\frac{\mathcal{E}_{\eta}(f)}{\left\Vert f\right\Vert _{\ell^{2}(G)}^{2}}\le\left|S\right|\left(\sum_{|g|\le r}|g|^{2}\eta(g)\right)\frac{\mathcal{E}_{\mu}(f)}{\left\Vert f\right\Vert _{\ell^{2}(G)}^{2}}+\eta\left(\left\{ g:|g|\ge r\right\} \right).
\]
Now we restrict to functions with $\left|\mbox{supp}f\right|\le e^{r}$,
then it follows that 
\begin{align*}
\Lambda_{\eta}(e^{r}) & \le\left|S\right|\left(\sum_{|g|\le r}|g|^{2}\eta(g)\right)\Lambda_{\mu}(e^{r})+\eta\left(\left\{ g:|g|\ge r\right\} \right)\\
 & \le C|S|r^{-2}\left(\sum_{|g|\le r}|g|^{2}\eta(g)\right)+\eta\left(\left\{ g:|g|\ge r\right\} \right).
\end{align*}
Since $\eta$ has finite first moment, $\sum_{g\in G}|g|\eta(g)<\infty$,
we have $\sum_{|g|\ge r}|g|\eta(g)\to0$ as $\mbox{ }r\to\infty$.
Therefore, $r\eta\left(\left\{ g:|g|\ge r\right\} \right)\to0$ as
$r\to\infty$, and 
\begin{align*}
\frac{1}{r}\sum_{g:|g|\le r}\left|g\right|^{2}\eta\left(g\right) & \le\frac{1}{r}\left(r^{1/2}\sum_{|g|\le r^{1/2}}|g|\eta(g)+r\sum_{|g|\ge r^{1/2}}|g|\eta(g)\right)\\
 & \le r^{-1/2}\sum_{g\in G}|g|\eta(g)+\sum_{|g|\ge r^{1/2}}|g|\eta(g),
\end{align*}
which implies 
\[
\frac{1}{r}\sum_{g:|g|\le r}\left|g\right|^{2}\eta\left(g\right)\to0\mbox{ as }r\to\infty.
\]
We conclude that $r\Lambda_{\eta}(e^{r})\to0$ as $r\to\infty$. 

\end{proof}

We now prove an explicit upper bound on entropy when the decay of
the return probability $\mu^{(2n)}(\id)$ is much slower than $\exp\left(-n^{\frac{1}{2}}\right).$

\begin{theorem}\label{bound}

Let $\mu$ be a symmetric probability measure of finite entropy on
$G$. Suppose $\mu^{(2n)}(\id)\ge\exp\left(-\gamma(n)\right)$ where
$\gamma:[1,\infty)\to\mathbb{R}_{+}$ is a function such that for
some $\beta\in\left(0,\frac{1}{2}\right)$, the function $n\to n^{\beta}/\gamma(n)$
is non-decreasing. Define 
\begin{equation}
\rho(n)=\inf\left\{ x:\ \gamma^{-1}(x/2)/x>n\right\} .\label{eq:rho}
\end{equation}
Then there is a constant $C=C(\beta)$ such that 
\[
H_{\mu}(n)\le C\left(\rho(n)+1\right).
\]

\end{theorem}

\begin{proof}

The proof is analogous to the proof of the \textquotedbl{}fundamental
inequality\textquotedbl{}, see for example \cite[Proposition 3.4]{Blach`ere2008}.

Let 
\[
k_{n}=\min\left\{ k:\ \frac{e^{k}}{\gamma^{-1}\left(e^{k}/2\right)}<\frac{1}{n}\right\} =\left\lceil \log\rho(n)\right\rceil .
\]
As before, let $\left(U_{k}\right)$ be a sequence of finite subsets
in $G$ such that $|U_{k}|=e^{e^{k}}$ and $\lambda_{\mu}(U_{k})\le2\Lambda_{\mu}\left(e^{e^{k}}\right)$.
Define $\Omega_{k}=\cup_{j=1}^{k}U_{j}^{-1}U_{j}$, $\Omega_{0}=\emptyset$.
Note that from definition of the set $\Omega_{k}$, we have 
\[
\left|\Omega_{k}\right|\le\sum_{j=1}^{k}\left|U_{j}\right|^{2}\le\frac{1}{1-e^{-2(e-1)}}e^{2e^{k}}.
\]
By (\ref{eq:tail}) and (\ref{eq:L-psi}), we have 
\begin{equation}
\mathbf{P}(W_{n}\notin\Omega_{k})\le\mathbf{P}(W_{n}\notin U_{k}^{-1}U_{k})\le\frac{ne^{k}}{4\gamma^{-1}\left(e^{k}/2\right)}.\label{eq:out-1}
\end{equation}
The entropy of $W_{n}$ can be decomposed as 
\begin{align*}
H_{\mu}(n) & =\mathbf{P}(W_{n}\in\Omega_{k_{n}})H\left(W_{n}|W_{n}\in\Omega_{k_{n}}\right)\\
 & +\sum_{k=k_{n}+1}^{\infty}\mathbf{P}\left(W_{n}\in\Omega_{k}\setminus\Omega_{k-1}\right)H\left(W_{n}|W_{n}\in\Omega_{k}\setminus\Omega_{k-1}\right)+\Delta_{n},
\end{align*}
where 
\[
\Delta_{n}=-\mathbf{P}(W_{n}\in\Omega_{k_{n}})\log\left(\mathbf{P}(W_{n}\in\Omega_{k_{n}})\right)-\sum_{k=k_{n}+1}^{\infty}\mathbf{P}(W_{n}\in\Omega_{k}\setminus\Omega_{k-1})\log\left(\mathbf{P}(W_{n}\in\Omega_{k}\setminus\Omega_{k-1})\right).
\]
By convexity of entropy, we have 
\begin{align*}
H_{\mu}(n) & -\Delta_{n}\le\log\left|\Omega_{k_{n}}\right|+\sum_{k\ge k_{n}+1}\log\left|\Omega_{k}\setminus\Omega_{k-1}\right|\mathbf{P}(W_{n}\in\Omega_{k}\setminus\Omega_{k-1})\\
 & \le c_{0}+2e^{k_{n}}+\sum_{k\ge k_{n}}2e^{k+1}\mathbf{P}(W_{n}\notin\Omega_{k}),
\end{align*}
where $c_{0}=\log\frac{1}{1-e^{-2(e-1)}}$. By (\ref{eq:out-1}),
\begin{equation}
\sum_{k\ge k_{n}}e^{k}\mathbf{P}(W_{n}\notin\Omega_{k})\le\sum_{k\ge k_{n}}e^{k}\frac{ne^{k}}{4\gamma^{-1}\left(e^{k}/2\right)}.\label{eq:out}
\end{equation}
From the assumption $n^{\beta}/\gamma(n)$ is non-decreasing and $\beta<\frac{1}{2}$,
we have that $\gamma^{-1}(t)/\gamma^{-1}(s)\ge\left(t/s\right)^{1/\beta}$.
Thus for all $k>k_{n}$, 
\[
\frac{e^{2k}}{\gamma^{-1}(e^{k}/2)}\le\frac{e^{2k_{n}}}{\gamma^{-1}(e^{k_{n}}/2)}\frac{e^{2k-2k_{n}}}{\left(e^{k}/e^{k_{n}}\right)^{1/\beta}}=\frac{e^{2k_{n}}}{\gamma^{-1}(e^{k_{n}}/2)}e^{-\left(1/\beta-2\right)(k-k_{n})}.
\]
Plug it into (\ref{eq:out}), we have 
\begin{align*}
\sum_{k\ge k_{n}}e^{k}\mathbf{P}(W_{n}\notin\Omega_{k}) & \le\frac{ne^{2k_{n}}}{\gamma^{-1}(e^{k_{n}}/2)}\sum_{k\ge k_{n}}e^{-(1/\beta-2)(k-k_{n})}\\
 & \le C_{\beta}\frac{ne^{2k_{n}}}{\gamma^{-1}(e^{k_{n}}/2)}\le C_{\beta}e^{k_{n}},
\end{align*}
where $C_{\beta}=\left(1-e^{-(1/\beta-2)}\right)^{-1}$. The last
step used the definition of $k_{n}$. By the definition of $\rho(n)$,
we have $\rho(n)\le e^{k_{n}}\le e\rho(n)$, we have that 
\begin{equation}
H_{\mu}(n)-\Delta_{n}\le c_{0}+2\left(1+C_{\beta}\right)e^{k_{n}}\le c_{0}+2e\left(1+C_{\beta}\right)\rho(n).\label{eq:H-rho}
\end{equation}
Finally, we show that $\Delta_{n}$ is bounded by a constant. By the
inequality $-a\log a\le2e^{-1}\sqrt{a}$, we have 
\[
\Delta_{n}\le2e^{-1}+2e^{-1}\sum_{k=k_{n}+1}^{\infty}\mathbf{P}(W_{n}\in\Omega_{k}\setminus\Omega_{k-1})^{\frac{1}{2}}.
\]
By (\ref{eq:out-1}) again, we have
\begin{align*}
\sum_{k=k_{n}+1}^{\infty}\mathbf{P}(W_{n} & \in\Omega_{k}\setminus\Omega_{k-1})^{\frac{1}{2}}\le\sum_{k=k_{n}+1}^{\infty}\left(\frac{ne^{k}}{4\gamma^{-1}\left(e^{k}/2\right)}\right)^{\frac{1}{2}}\\
 & \le\left(\frac{ne^{k_{n}}}{4\gamma^{-1}\left(e^{k_{n}}/2\right)}\right)^{\frac{1}{2}}\sum_{k=k_{n}+1}^{\infty}e^{-1/2\left(1/\beta-1\right)(k-k_{n})}\le2C_{\beta}.
\end{align*}
We conclude that $\Delta_{n}\le2e^{-1}(1+2C_{\beta}).$ The statement
follows from (\ref{eq:H-rho}).

\end{proof}

\begin{remark}\label{profile}

From the proof it is clear that Theorem \ref{bound} can be rephrased
in terms of the spectral profile as follows. Let $\mu$ be a symmetric
probability measure of finite entropy on $G$. Suppose there exists
$\alpha>1$ such that $\frac{1}{x^{\alpha}\Lambda_{\mu}(e^{x})}$
is non-decreasing. Define 
\[
\tilde{\rho}(n)=\inf\left\{ x:\ \Lambda_{\mu}(e^{x})\le\frac{1}{n}\right\} ,
\]
then there is a constant $C=C(\alpha)$ such that 
\[
H_{\mu}(n)\le C(\tilde{\varrho}(n)+1).
\]

\end{remark}

\begin{remark}

Suppose $\gamma(n)$ is a regularly varying function $\gamma(n)=n^{\beta}\ell(n)$
where $\beta\in\left(0,1/2\right)$, and $\ell$ is a slowly varying
function satisfying $\ell(n^{b})\simeq\ell(n)$ for all $b>0$. Then
by standard asymptotic inversion, see \cite[Proposition 1.5.15]{BGT},
we have that the function $\rho$ defined in (\ref{eq:rho}) satisfies
\[
\rho(n)\simeq n^{\frac{\beta}{1-\beta}}\ell^{\frac{1}{1-\beta}}(n).
\]
In particular, Theorem \ref{bound-1} in the Introduction follows
from Theorem \ref{bound}.

\end{remark}

\begin{example}

Consider simple random walk on the wreath product $\mathbb{Z}\wr\mathbb{Z}$.
By Pittet and Saloff-Coste \cite{Pittet2002}, 
\[
\mu^{(2n)}(e)\simeq\exp\left(-n^{\frac{1}{3}}\log^{\frac{2}{3}}n\right).
\]
From Erschler \cite{Erschlerdrift}, we have $H_{\mu}(n)\simeq n^{\frac{1}{2}}\log n$.
Theorem \ref{bound} implies that $H_{\mu}(n)\le n^{\frac{1}{2}}\log n$,
which is a sharp upper bound. We will see in Section \ref{sec:bubble}
a family of groups where Theorem \ref{bound} provides sharp upper
bounds on entropy. 

\end{example}

\section{Upper bounds on rate of escape\label{sec:escape}}

In this section, we show that if there is some additional information
on the spectral profile of balls, then the Markov type inequality
provides an upper bound on the rate of escape. 

We first show that Lemma \ref{tail} can be strengthened to provide
the following bound. 

\begin{lemma}\label{max-tail}

Let $\mu$ be a symmetric probability measure on $G$, and $U$ be
a finite subset of $G$. Then for a $\mu$-random walk $\left(W_{n}\right)$
on $G$, 
\[
\mathbf{P}\left(\exists k:\ 0\le k\le n,\ W_{k}\notin U^{-1}U\right)\le(32n+1)\lambda_{\mu}(U).
\]

\end{lemma}

\begin{proof}

Together with the classical Doob's $L^{2}$-maximal inequality for
submartingles, the proof of the Markov type inequality (\ref{eq:markov})
in \cite[Theorem 2.1]{Naor2008} yields the maximal inequality: for
any $f\in Z^{1}(G,\pi)$ for some representation $\pi:G\to\mbox{Isom}(\mathcal{H})$,
\begin{equation}
\mathbf{E}\left[\max_{0\le k\le n}\|f(W_{k})\|_{\mathcal{H}}^{2}\right]\le(32n+1)\mathbf{E}\left[\|f(W_{1})\|_{\mathcal{H}}^{2}\right].\label{eq: max}
\end{equation}
To see this, let $v=\mathbf{E}[f(W_{1})]$, $X_{1},X_{2},\ldots$
be the random walk increments, then by \cite[Equation (20)]{Naor2008},
\[
2f(W_{n})=M_{n}-\pi(W_{n})N_{n}+\pi(W_{n})v,
\]
where $\left(M_{k}\right)_{k=0}^{n}$ is a martingle with respect
to the filtration induced by $\left(X_{k}\right){}_{k=0}^{n}$, and
$\left(N_{k}\right){}_{k=0}^{n}$ is a martingale with respect to
$\left(X_{n-k}\right)_{k=0}^{n}$, and $M_{k}$ has the same law as
$N_{k}$. Therefore

\begin{align*}
4\mathbf{E}\left[\max_{0\le k\le n}\|f(W_{k})\|_{\mathcal{H}}^{2}\right] & \le3\left(\mathbf{E}\left[\max_{0\le k\le n}\|M_{k}\|_{\mathcal{H}}^{2}\right]+\mathbf{E}\left[\max_{0\le k\le n}\|N_{k}\|_{\mathcal{H}}^{2}\right]+\left\Vert v\right\Vert _{\mathcal{H}}^{2}\right)\\
 & \le24\mathbf{E}\left[\|M_{n}\|_{\mathcal{H}}^{2}\right]+3\left\Vert v\right\Vert _{\mathcal{H}}^{2}.
\end{align*}
In the last step we used Doob's $L^{2}$-maximal inequality for submartingales
that $\mathbf{E}\left[\max_{0\le k\le n}\|M_{k}\|_{\mathcal{H}}^{2}\right]\le4\mathbf{E}\left[\|M_{n}\|_{\mathcal{H}}^{2}\right]$.
By \cite[Equation (22)]{Naor2008}, we have
\[
\mathbf{E}\left[\|M_{n}\|_{\mathcal{H}}^{2}\right]\le4n\mathbf{E}\left[\|f(W_{1})\|_{\mathcal{H}}^{2}\right].
\]
The inequality (\ref{eq: max}) follows. 

Let the test function $\phi$ and the associated $1$-coboundary $\sigma$
be chosen as in the proof of Lemma \ref{tail}. Recall that if $g\notin U^{-1}U$,
then $\mbox{supp }\phi\cap\mbox{supp }\tau_{g}\phi=\emptyset$, and
it follows that 
\[
\|\sigma(g)\|_{2}^{2}=\|\phi\|_{2}^{2}+\|\tau_{g}\phi\|_{2}^{2}=2\|\phi\|_{2}^{2}.
\]
Therefore 
\begin{align*}
2\|\phi\|_{2}^{2}\mathbf{P}\left(\exists k:\ 0\le k\le n,\ W_{k}\notin U^{-1}U\right) & \le\mathbf{E}\left[\max_{0\le k\le n}\|\sigma(W_{k})\|_{2}^{2}\right]\\
 & \le2(32n+1)\mathcal{E}_{\mu}\left(\phi\right).
\end{align*}
The statement of the lemma follows from the choice of $\phi$.

\end{proof}

Let $\mu$ be a symmetric probability measure on $G$. Write $B(r)=B(\id,r)$.
Consider the spectral profile of balls 
\[
\lambda_{\mu}\left(B(r)\right)=\inf\{\mathcal{E}_{\mu}(f):\ \mbox{support}(f)\subset B(e,r),\|f\|_{2}=1\}.
\]

\begin{theorem}\label{moment}

Let $\mu$ be a symmetric probability measure on $G$. Suppose there
exists a function $f$ and constant $\theta>0$ such that 
\[
\lambda_{\mu}(B(r))\le f(r)\mbox{ and }\frac{f(r)}{f(s)}\le C_{0}\left(\frac{r}{s}\right)^{-\theta}\ \mbox{for all }r\ge s>0.
\]
Then for any $\alpha\in(0,\theta)$, there exists a constant $C=C(\theta,\alpha,C_{0})>0$
such that the $\mu$-random walk on $G$ satisfies 
\[
\mathbf{E}\left[\max_{0\le k\le n}\left|W_{k}\right|^{\alpha}\right]\le C\varrho(n)^{\alpha},
\]
where 
\[
\varrho(n):=\inf\left\{ r>0:\ f(r)<\frac{1}{n}\right\} .
\]

\end{theorem}

\begin{proof}

The proof is along the same lines as Theorem \ref{bound}. Let $k_{n}=\inf\left\{ k:\ f\left(e^{k}\right)<1/n\right\} ,$
note that by definitions $e^{k_{n}-1}<\varrho(n)\le e^{k_{n}}$. We
have 
\[
\mathbf{E}\max_{0\le k\le n}\left|W_{k}\right|^{\alpha}\le e^{\alpha(k_{n}+1)}+\sum_{k\ge k_{n}+1}e^{\alpha(k+1)}\mathbf{P}\left(\exists\ell\le n:\ W_{\ell}\notin B\left(e^{k+1}\right)\setminus B\left(e^{k}\right)\right).
\]
By Lemma \ref{max-tail}, we have 
\[
\mathbf{P}\left(\exists\ell\le n:\ W_{\ell}\notin B\left(e^{k}\right)\right)\le(32n+1)\lambda_{\mu}\left(B\left(e^{k-1}\right)\right).
\]
Therefore 
\begin{align*}
\mathbf{E}\max_{0\le k\le n}\left|W_{k}\right|^{\alpha} & \le e^{\alpha k_{n}}+\sum_{k\ge k_{n}}e^{\alpha(k+1)}(32n+1)\lambda_{\mu}\left(B\left(e^{k-1}\right)\right)\\
 & \le e^{\alpha(k_{n}+1)}+33e^{2\alpha}n\sum_{k\ge k_{n}}e^{\alpha k}f(e^{k})\\
 & =e^{\alpha(k_{n}+1)}+33e^{2\alpha}nf(e^{k_{n}})\sum_{k\ge k_{n}}e^{\alpha k}\frac{f(e^{k})}{f(e^{k_{n}})}\\
 & \le e^{\alpha(k_{n}+1)}+e^{2\alpha}C_{0}e^{\alpha k_{n}}\sum_{k\ge k_{n}}e^{-(\theta-\alpha)(k-k_{n})}\\
 & =\left(e^{\alpha}+\frac{C_{0}e^{2\alpha}}{1-e^{-(\theta-\alpha)}}\right)e^{\alpha k_{n}}\\
 & \le C\rho(n)^{\alpha}.
\end{align*}

\end{proof}

We now review some examples from the literature where symmetric simple
random walks are known to satisfy $\lambda_{\mu}(B(r))\simeq r^{-2}$.
By the universal diffusive lower bound from Lee and Peres \cite{Lee2013}
and Theorem \ref{escape}, we know that in these examples, symmetric
simple random walks exhibit diffusive rate of escape, that is $L_{\mu}(n)\simeq n^{\frac{1}{2}}$. 

\begin{definition}[Tessera \cite{Tessera2013}]

The class $\mathcal{S}$ of finitely generated geometrically elementary
solvable groups is the smallest class of finitely generated groups
such that 
\begin{description}
\item [{(i)}] $\mathcal{S}$ contains all the finitely generated closed
subgroups of $T(d,k)$ for any integer $d$ and local field $k$,
where $T(d,k)$ is the group of invertible upper triangular matrices
of size $d$ in a local field $k$;
\item [{(ii)}] $\mathcal{S}$ is stable under taking finite direct products,
finitely generated subgroups and quotients;
\item [{(iii)}] $\mathcal{S}$ is stable under quasi-isometry.
\end{description}
\end{definition}

The collection $\mathcal{S}$ is exactly the subset of finitely generated
groups in the class GES in \cite{Tessera2013}. Our results actually
extend to compactly generated locally compact unimodular groups, but
for simplicity we restrict ourselves to discrete groups here. 

\begin{example}

The following examples are known to be in $\mathcal{S}$, see Tessera
\cite{Tessera2013}.
\begin{itemize}
\item finitely generated nilpotent groups, more generally polycyclic groups,
\item lamplighter groups $F\wr\mathbb{Z}$, where $F$ is a finite group,
\item solvable Baumslag-Solitar groups $BS(1,q)=\left\langle t,x|txt^{-1}=x^{q}\right\rangle $.
\item torsion free solvable groups of finite Prüfer rank. Recall that a
group has finite Prüfer rank if there is an integer $r$ such that
any of its finitely generated subgroup admits a generating set of
cardinality less or equal to $r$. Return probability and the spectral
profile of simple random walks on torsion free solvable groups of
finite Prüfer rank are also treated in Pittet and Saloff-Coste \cite{PSCrank},
where the authors proved the lower bound $\mu^{(2n)}(\id)\gtrsim e^{-n^{1/3}}$
for $\mu$ a symmetric probability measure with finite generating
support on the group.
\end{itemize}
\end{example}

Note that the class $\mathcal{S}$ is defined in both algebraic and
geometric terms. It contains non-virtually solvable groups such as
$F\wr\mathbb{Z}$ where $F$ is a non-solvable finite group. It also
contains finitely generated groups that are not residually finite,
in particular not linear, see \cite{Tessera2013}.

\begin{corollary}\label{diffusive}

Let $G$ be a group in the class $\mathcal{S}$. Then for any symmetric
probability measure $\mu$ with finite second moment on $G$, there
exists a constant $C=C(\mu,G)$ such that for any $\alpha\in[1,2)$,
\[
\mathbf{E}\max_{0\le k\le n}\left|W_{k}\right|^{\alpha}\le\frac{C}{2-\alpha}n^{\frac{\alpha}{2}}.
\]

\end{corollary}

\begin{proof}

The main result of \cite{Tessera2013} states that on a geometrically
elementary solvable group, $\lambda_{\mu}(B(r))\simeq r^{-2}$ for
a symmetric probability measure $\mu$ of finite generating support.
By standard comparison of Dirichlet forms \cite{PSCstab}, the same
estimates holds for a symmetric probability measures $\mu'$ with
finite second moment. Then Theorem \ref{moment} applies.

\end{proof}

\section{Sharpness of the entropy bounds\label{sec:bubble}}

To illustrate the sharpness of the bound in Theorem \ref{bound} when
the decay of return probability is away from the critical behavior
$\exp(-n^{\frac{1}{2}})$, we consider the family of bubble groups
as in Kotowski and Vir\'{a}g \cite{Kotowski2015}. The bubble groups
were first introduced by Amir and Kozma \cite{amir2016groups} in
their study of growth of harmonic functions on groups. 

\begin{figure}
\includegraphics[scale=0.3]{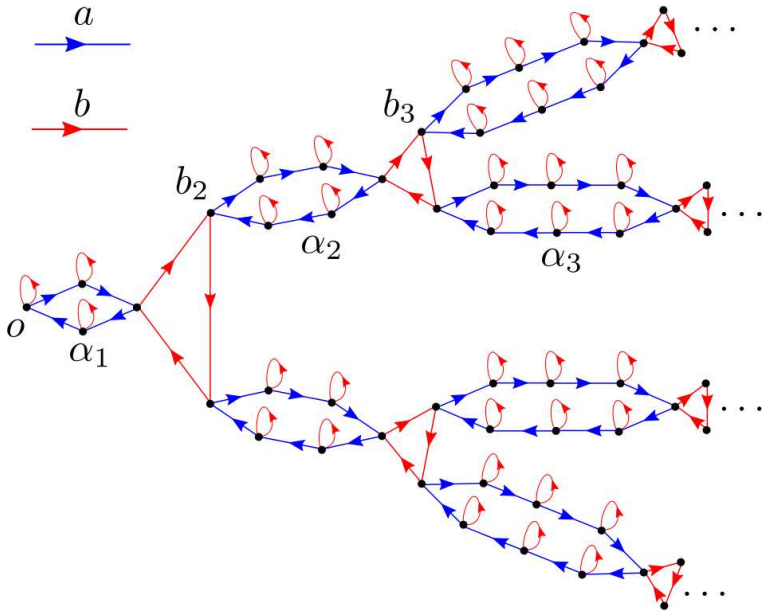}

\caption{the first 3 levels of the Schreier graph, picture from M. Kotowski
and B. Virág \cite{Kotowski2015}}
\end{figure}

The bubble groups are defined via their actions on certain marked
graphs. We first describe these graphs. Let $\mathbf{a}=\left(\alpha_{1},\alpha_{2},\ldots\right)$
be a sequence of natural numbers, $\mathbf{T}$ be a rooted tree with
forward degree sequence $(1,2,2,\ldots)$, that is except the root,
every vertex has two children. The bubble graph $\mathcal{M}_{\mathbf{a}}$
is obtained as follows. Each edge at level $k\ge1$ (we make the convention
that the level of an edge is the level of the child on that edge)
is replaced by a cycle of length $2\alpha_{k}$ called a bubble. Each
vertex at level $k\ge1$ (we ignore the root which is now part of
a cycle of length $2\alpha_{1}$) is blown-up to a $3$-cycle. These
$3$-cycles are called branching cycles. Finally, at each vertex which
belongs only to a bubble (but not to a branching cycle), we add a
self loop. Having chosen an orientation along each cycle (say, clockwise),
we label each edge of the bubble with the letter $a$, each edge of
the branching cycle with the letter $b$, and the self loops at vertices
that belong only to a bubble are labeled with the letter $b$. Figure
1 represents the first $3$ levels of the Schreier graph $\mathcal{M}_{\mathbf{a}}$,
with $\alpha_{1}=2,\alpha_{2}=3$ and $\alpha_{3}=4$.

The bubble group $\Gamma_{\mathbf{a}}$ is a subgroup of the permutation
group of the vertex set of $X_{\mathbf{a}}$ generated by two elements
$a$ and $b$. The action of the permutation $a$ (resp. $b$) on
any vertex $x$ in $\mathcal{M}_{\mathbf{a}}$ is indicated by the
oriented labeled edge at $x$ marked with an $a$ (resp. a $b$).
Informally, $a$ rotates the long bubbles whereas $b$ rotates the
branching cycles. We write the left group action of $g$ on vertex
$x$ as $g\cdot x$. 

Consider the permutational wreath product of the bubble group $\Gamma_{\mathbf{a}}$
with $\mathbb{Z}$ over the Schreier graph $\mathcal{M}_{\mathbf{a}}$,
$G_{\mathbf{a}}=\mathbb{Z}\wr_{\mathcal{M}_{\mathbf{a}}}\Gamma_{\mathbf{a}}$.
Formally, $G$ is the semi-direct product $\left(\oplus_{\mathcal{M}_{\mathbf{a}}}\mathbb{Z}\right)\rtimes\Gamma_{\mathbf{a}}$,
where $\Gamma_{\mathbf{a}}$ acts by permuting the coordinates. A
group element of $G$ is recorded as $(f,\gamma)$, where $f:\mathcal{M}_{\mathbf{a}}\to\mathbb{Z}$
is a function of finite support, and $\gamma\in\Gamma_{\mathbf{a}}$.
Multiplication is given by 
\[
(f,\gamma)(f',\gamma')=(f+\tau_{\gamma}f',\gamma\gamma'),
\]
where $\tau_{\gamma}$ is the translation $\tau_{\gamma}f(x)=f(\gamma^{-1}\cdot x)$
and we use additive notation for $\mathbb{Z}$. We identify $\Gamma_{\mathbf{a}}$
as a subgroup of $G_{\mathbf{a}}$ by the embedding $\gamma\mapsto(\mathbf{0},\gamma)$
for every $\gamma\in\Gamma_{\mathbf{a}}$ and $\mathbb{Z}$ as a subgroup
of $G_{\mathbf{a}}$ by $z\mapsto\left(z\delta_{0},\id_{\Gamma_{\mathbf{a}}}\right)$.

Let $\nu$ be the symmetric probability measure on $\Gamma_{\mathbf{a}}$
defined as 
\[
\nu(a^{\pm1})=\nu(b^{\pm1})=\frac{1}{4}.
\]
Let $\eta$ be the symmetric probability measure on $\mathbb{Z}$
defined as $\eta(\pm1)=\frac{1}{2}$. Consider the \textquotedbl{}switch-walk-switch\textquotedbl{}
measure $\mu=\eta\ast\nu\ast\eta$ on the permutational wreath product
$G_{\mathbf{a}}$. The $\mu$-random walk on $G_{\mathbf{a}}$ can
be described as follows. Let $X_{1},X_{2},\ldots$ be a sequence of
i.i.d.\ random variables on $\Gamma_{\mathbf{a}}$ with distribution
$\nu$ and $Z_{1},Z_{2},\ldots$ be a sequence of i.i.d.\ random
variables on $\mathbb{Z}$ with distribution $\eta$. Then the random
variable $W_{n}=Z_{1}X_{1}Z_{2}\ldots Z_{2n-1}X_{n}Z_{2n}$ has distribution
$\mu^{(n)}$. Write $W_{n}=(L_{n},Y_{n})$, where $L_{n}$ denotes
the lamp configuration at time $n$ and $Y_{n}$ denotes the position
of the walker on the Cayley graph of $\Gamma_{\mathbf{a}}$. Then
we have 
\begin{align*}
Y_{n} & =X_{1}\ldots X_{n},\\
L_{n}(x) & =Z_{1}\mathbf{1}_{\{x=o\}}+\sum_{j=1}^{n-1}\left(Z_{2j}+Z_{2j+1}\right)\mathbf{1}_{\left\{ x=Y_{j}\cdot o\right\} }+Z_{2n}\mathbf{1}_{\left\{ x=Y_{n}\cdot o\right\} }.
\end{align*}
An important point is that the $\mathbb{Z}$-lamp configurations are
updated along the inverted orbit $\left(Y_{j}\cdot o\right)_{j=0}^{n}$.
More information about permutational wreath products and random walks
with \textquotedbl{}switch-walk-switch\textquotedbl{} step distribution
on them can be found in \cite{Amir2012,Kotowski2015} . 

Whether $H_{\mu}(n)$ grows linearly or not is closely related to
the transience/recurrence property of the underlying bubble graph.
Indeed, when the induced random walk $\left(Y_{n}^{-1}\cdot o\right)$
is transient on $\mathcal{M}_{\mathbf{a}}$, the Poisson boundary
of $(G,\mu)$ is non-trivial, see \cite[Theorem 5.1]{Kotowski2015}
and also Bartholdi and Erschler \cite[Section 3]{Bartholdi2016}.
Note that on the bubble graph, the effective resistance from the root
$o$ to the level $k$ branching cycles is quite easy to calculate.
The graph $\mathcal{M}_{\mathbf{a}}$ is recurrent if and only if
\[
\sum_{k=1}^{\infty}\frac{\alpha_{k}}{2^{k}}=\infty.
\]

We now apply the method in Amir and Virág \cite{Amir2012} to give
an explicit entropy lower bound for the $\mu$-random walk on $\mathbb{Z}\wr_{\mathcal{M}_{\mathbf{a}}}\Gamma_{\mathbf{a}}$,
in the case that $\inf_{k}\alpha_{k+1}/\alpha_{k}>2.$ The basic idea
for the lower bound is to collect contribution to entropy from the
lamp configurations over the Schreier graph. The permutational wreath
extension with $\mathbb{Z}$ plays an important role in both the entropy
lower bound and return probability upper bound. Let $Q_{n}$ be the
occupation time measure of the inverted orbit $\left(Y_{j}\cdot o\right)_{j=0}^{n}$
in the first $n$ steps. Namely, for each $x\in\mathcal{M}_{\mathbf{a}}$,
$Q_{n}(x)$ denotes the number of visits of the inverted orbit to
the vertex $x$ in the first $n$ steps,
\begin{equation}
Q_{n}(x)=\sum_{j=0}^{n}\mathbf{1}_{\left\{ x=Y_{j}\cdot o\right\} }.\label{eq:Q}
\end{equation}

\begin{lemma}\label{bubble-entropy}

Let $\mathbf{a}$ be a scaling sequence satisfying 
\[
\inf_{k}\alpha_{k+1}/\alpha_{k}>2.
\]
Then for the $\mu$-random walk on $G_{\mathbf{a}}$ described above,
there exists a constant $c>0$ such that 
\[
H_{\mu}(n)\ge c\left|B_{\mathcal{M}_{\mathbf{a}}}(o,\sqrt{n})\right|\log n.
\]

\end{lemma}

\begin{proof}

Denote by $T$ the first time that the induced random walk returns
to $o$,
\[
T=\min\left\{ n>0:\ Y_{n}\cdot o=o\right\} =\min\{n>0:\ Y_{n}^{-1}\cdot o=o\}.
\]
For simple random walk $\nu$ on $\mathbb{Z}$ we have $H_{\nu}(k)\ge c_{0}\log k$.
Let $Q_{n}$ be defined as in (\ref{eq:Q}), then 
\begin{equation}
H_{\mu}(n)\ge H\left(W_{n}|Q_{n}\right)\ge c_{0}\mathbf{E}\left[\sum_{x\in X_{\mathbf{a}}}\log\left(1+Q_{n}(x)\right)\right].\label{eq:H-Q}
\end{equation}
From the proof of \cite[Theorem 9]{Amir2012}, we have that for any
sub-additive function $f:[0,\infty)\to[0,\infty)$, 
\begin{equation}
\mathbf{E}\left[\sum_{x\in X_{\mathbf{a}}}f\left(Q_{n}(x)\right)\right]\ge\frac{npf\left(\frac{1}{p}\right)}{16}\label{eq:f-Qn}
\end{equation}
for all $0<p\le\mathbf{P}(T>n)$. Next we apply (\ref{eq:f-Qn}) with
$f(x)=\log(1+x)$. By (\ref{eq:H-Q}), a lower estimate on $\mathbf{P}(T>n)$
will give a lower bound for entropy. We now follow the idea in the
proof of \cite[Proposition 18]{Amir2012}. 

Let $S_{r}=\{x\in\mathcal{M}_{\mathbf{a}}:\ d(o,x)=r\}$ be the set
of vertices that are distance $r$ from the root $o$. Let $\sigma_{r}$
denote the hitting time of the set $S_{r}$, and $T_{r}^{\ast}$ be
the first return to $o$ after hitting $S_{r}$, that is $T_{r}^{*}=\inf\{n:\ Y_{n}^{-1}\cdot o=o,\ n>\sigma_{r}\}.$
Then on the event that the inverted orbit $\{Y_{j}^{-1}\cdot o\}_{j\ge1}$
hits $S_{r}$ before $o$ and the first return time to $o$ after
hitting $S_{r}$ is later than $n$, we have the first return to $o$
is after time $n$. In other words we have inclusion of events:
\[
\{T>n\}\supseteq\left\{ \{Y_{j}^{-1}\cdot o\}_{j\ge1}\ \mbox{hits }S_{r}\mbox{ before }o\right\} \cap\{T_{r}^{*}>n\}.
\]
Let $\mathcal{R}(o\leftrightarrow S_{r})$ be the effective resistance
between the root $o$ and the level set $S_{r}$, then 
\[
\mathbf{P}\left(\{Y_{j}^{-1}\cdot o\}_{j\ge1}\ \mbox{hits }S_{r}\mbox{ before }o\right)=\frac{1}{2\mathcal{R}(o\leftrightarrow S_{r})}.
\]
We have for $r=\ell+\sum_{j=1}^{k}\alpha_{j}+k$, where $0\le\ell\le\alpha_{k+1}$,

\[
\mathbf{P}\left(\{Y_{j}^{-1}\cdot o\}_{j\ge1}\ \mbox{hits }S_{r}\mbox{ before }o\right)=\frac{1}{2\left(2^{-k-1}\ell+\sum_{j=1}^{k}2^{-j}\alpha_{j}\right)}\ge\frac{1}{2^{-k+2}(\alpha_{k}+\ell)}.
\]
In the last step we used the assumption that $\inf_{k}\alpha_{k+1}/\alpha_{k}>2.$ 

Let $\tilde{\sigma}_{\ell}$ denote the hitting time of $0$ of lazy
simple random walk on the interval $[0,\ell]\cap\mathbb{Z}$ with
reflecting boundary, starting from the end $\ell$. Then $T_{r}^{*}$
stochastically dominates $\tilde{\sigma}_{\ell}+\tilde{\sigma}_{\alpha_{k}}+...+\tilde{\sigma}_{\alpha_{1}}$.
Therefore from classical estimates of lazy simple random walk on an
interval, we have that there exists an absolute constant $c_{1}>0$
such that 

\[
P\left(T_{r}^{\ast}>c_{1}\left(\alpha_{k}^{2}+\ell^{2}\right)\right)\ge\frac{1}{4}.
\]
From the Markov property, we have 
\[
P\left(T>c\left(\alpha_{k}^{2}+\ell^{2}\right)\right)\ge\frac{1}{4}\frac{1}{2^{-k+2}(\alpha_{k}+\ell)}=\frac{2^{k}}{16\left(\alpha_{k}+\ell\right)}.
\]
For time $n=c\left(\alpha_{k}^{2}+\ell^{2}\right)$, take $p_{n}=\frac{2^{k}}{16\left(\alpha_{k}+\ell\right)}$,
by (\ref{eq:f-Qn}) 
\begin{equation}
H_{\mu}(n)\ge c_{0}\mathbf{E}\left[\sum_{x\in X_{\mathbf{a}}}\log\left(1+Q_{n}(x)\right)\right]\ge\frac{c_{0}np_{n}\log\left(1+\frac{1}{p_{n}}\right)}{16}.\label{eq:Qn-p}
\end{equation}
Therefore for $n\in\left(c\alpha_{k}^{2},c\alpha_{k}^{2}+c\alpha_{k+1}^{2}\right),$
\begin{equation}
H_{\mu}(n)\ge c_{0}'n\left(\frac{2^{k}}{\sqrt{n}}\right)\log\left(\frac{\sqrt{n}}{2^{k}}\right).\label{eq:h-lower}
\end{equation}
Note that the volume of the ball $B_{\mathcal{M}_{\mathbf{a}}}(o,\sqrt{n})$
is comparable to $2^{k}\sqrt{n}$ for $n\in\left(c\alpha_{k}^{2},c\alpha_{k}^{2}+c\alpha_{k+1}^{2}\right)$,
the statement follows from (\ref{eq:h-lower}).

\end{proof}

The spectral profile $\Lambda_{\mu,G_{\mathbf{a}}}$ is estimated
in \cite[Theorem 5.8]{perm}. In particular, if there exists a constant
$\theta>1$ such that the scaling sequence $\mathbf{a}$ satisfies
\[
\inf_{k}\frac{\alpha_{k+1}}{\alpha_{k}}\ge\theta,
\]
then there exists a constant $C=C(\theta)>0$ such that 
\[
\frac{1}{Cr^{2}}\le\Lambda_{\mu,G_{\mathbf{a}}}\left(\left|B_{\mathcal{M}_{\mathbf{a}}}(o,r)\right|!\right)\le\frac{C}{r^{2}}.
\]
It follows from the Coulhon-Grigor'yan theory \cite{CNash,CG} that
in this case, the decay of the return probability satisfies
\begin{equation}
\mu^{(2n)}(\id)\simeq\exp\left(-\frac{n}{\phi(n)^{2}}\right),\label{eq:bubble-return}
\end{equation}
where $\phi$ is the inverse of the function $r\to r^{2}\left|B_{\mathcal{M}_{\mathbf{a}}}(o,r)\right|\log r$. 

In the special case where $\theta>2$, we have 
\[
\frac{n}{\phi(n)^{2}}\le n^{\frac{\theta+1}{3\theta+1}}(\log n)^{\frac{2\theta}{3\theta+1}}.
\]
Therefore the assumptions of Theorem \ref{bound} are satisfied, we
can deduce an entropy upper bound from (\ref{eq:bubble-return}).
In fact it is easier to use the upper bound on the spectral profile
directly, see Remark \ref{profile}. From the bound $\Lambda_{\mu,G_{\mathbf{a}}}\left(\left|B_{\mathcal{M}_{\mathbf{a}}}(o,r)\right|!\right)\le\frac{C}{r^{2}}$,
we have that 
\[
\tilde{\rho}(n)=\inf\left\{ x:\ \Lambda_{\mu}(e^{x})\le\frac{1}{n}\right\} \le\log\left(\left|B_{\mathcal{M}_{\mathbf{a}}}\left(o,\sqrt{Cn}\right)\right|!\right).
\]
Therefore by Theorem \ref{bound}, we have

\[
H_{\mu}(n)\le\log\left(\left|B_{\mathcal{M}_{\mathbf{a}}}(o,\sqrt{Cn})\right|!\right)\le C'\left|B_{\mathcal{M}_{\mathbf{a}}}(o,\sqrt{n})\right|\log n.
\]
Comparing the bound to Lemma \ref{bubble-entropy}, we see that when
the scaling sequence $\mathbf{a}$ satisfies $\inf_{k}\alpha_{k+1}/\alpha_{k}>2$,
the entropy upper bound deduced from return probability by Theorem
\ref{bound} is sharp. 

\begin{proof}[Proof of Proposition \ref{bubble-exponent}]

The wreath product $\mathbb{Z\wr\mathbb{Z}}$ satisfies the statement
with $\beta=1/3$.

Take $\alpha_{k}=2^{\theta k}$ with $\theta\in(1,\infty)$ as the
scaling sequence in the bubble graph $\mathcal{M}_{\mathbf{a}}$.
Then by \cite[Theorem 5.8]{perm}, see also \cite[Example 5.11]{perm},
on the group $G_{\mathbf{a}}=\mathbb{Z}\wr_{\mathcal{M}_{\mathbf{a}}}\Gamma_{\mathbf{a}}$,
we have
\[
\mu^{(2n)}(\id)\simeq\exp\left(-n^{\frac{\theta+1}{3\theta+1}}(\log n)^{\frac{2\theta}{3\theta+1}}\right).
\]
By Lemma \ref{bubble-entropy}, we have
\[
H_{\mu}(n)\ge\frac{c(\theta-1)}{2\theta}n^{\frac{\theta+1}{2\theta}}\log n.
\]
Using Theorem \ref{bound}, we obtain an upper bound on $H_{\mu}(n)$
from the lower bound on $\mu^{(2n)}(\id)$, it matches the entropy
lower bound above.

\end{proof}

\section{Extension to random walks on transitive graphs}

In this section we explain how to extend Theorem \ref{liouville},
\ref{bound-1} and \ref{escape} to symmetric random walks on transitive
graphs. Let $X=(V,E)$ be a vertex transitive graph, which is locally
finite, infinite and connected. By $\mbox{Aut}(X)$ we denote the
full automorphism group of $X$. The group $\mbox{Aut}(X)$ is locally
compact, equipped with the topology of pointwise convergence. In what
follows, let $\Gamma$ be a closed subgroup of $\mbox{Aut}(X)$ that
acts transitively on $X$. 

We say a Markov operator $P$ on $X$ is \emph{space homogenous} if
it is invariant under $\Gamma$, that is, 
\[
P(g\cdot x,g\cdot y)=P(x,y)\ \mbox{for all }g\in\Gamma,x,y\in X.
\]
Further we may assume that $P$ is irreducible. The Poisson boundary
and entropy theory of space homogenous Markov chains was systematically
developed in Kaimanovich and Woess \cite{Kaimanovich2002}. Let $H_{P}(n)$
be the Shannon entropy of the $n$-step transition probability,
\[
H_{P}(n)=-\sum_{x\in V}P^{n}(o,x)\log P^{n}(o,x).
\]
By \cite{Kaimanovich2002}, for transition operator of finite entropy
$H_{P}(1)<\infty$, the asymptotic entropy $h(P)=\lim_{n\to\infty}H_{P}(n)/n$
exists, and $(X,P)$ has the Liouville property if and only if $h(P)=0$. 

Fix a reference point $o\in X$. We will only consider symmetric Markov
operators, $P(x,y)=P(y,x)$ for all $x,y\in V$. 

\begin{theorem}\label{transitive}

Let $P$ be a space homogenous symmetric Markov chain on $X$ such
that $H_{P}(1)<\infty$. Theorem \ref{liouville} and \ref{bound}
hold with $\mu^{(2n)}(\id)$ replaced by $P^{(2n)}(o,o)$ and $H_{\mu}(n)$
replaced by $H_{P}(n)$. 

\end{theorem}

For each vertex $x\in X$, let $\Gamma_{x}$ be its stabilizer, $\Gamma_{x}=\{g\in\Gamma,\ g\cdot x=x\}$.
Write $K=\Gamma_{o}$. Since $K$ is open and compact, we can normalize
the left-invariant Haar measure $m_{\Gamma}$ on $\Gamma$ such that
$m_{\Gamma}(K)=1$. Denote by $\hat{m}_{\Gamma}$ the involution of
$m_{\Gamma}$, that is $\hat{m}_{\Gamma}(A)=m_{\Gamma}(A^{-1})$.
The \emph{modular function} $\Delta$ of the group $\Gamma$ is defined
as 
\[
\Delta(g)=\Delta_{\Gamma}(g)=\frac{dm_{\Gamma}}{d\hat{m}_{\Gamma}}(g).
\]
We say the group $\Gamma$ is \emph{unimodular} if $\Delta\equiv1$. 

For a vertex $x$ in the graph $X$, consider the coset of $K$ in
$\Gamma$ that moves $o$ to $x$, $\{g\in\Gamma,\ g\cdot o=x\}=g_{x}K$,
where $g_{x}\cdot o=x$. Since the graph $X$ is assumed to be locally
finite and connected, we have that $\left\{ g_{x}K:x\sim o\mbox{ or }x=o\right\} $
generates the group $\Gamma$. 

By \cite[Proposition 2.15]{Kaimanovich2002}, there is a one-to-one
correspondence between $\Gamma$-invariant Markov operators on $X$
and bi-$K$-invariant probability measures on $\Gamma$. Namely, given
a bi-$K$-invariant probability measure $\mu$ on $\Gamma$, define
$P_{\mu}$ as 
\[
P_{\mu}(g\cdot o,h\cdot o)=\mu(g^{-1}hK),
\]
then $P_{\mu}$ is a $\Gamma$-invariant Markov operator on $X$.
Note that $P_{\mu}$ is well-defined because $\mu$ is bi-$K$-invariant.
In the other direction, set 
\begin{equation}
d\mu(g)=P(o,g\cdot o)dm_{\Gamma}(g),\label{eq:mu-P}
\end{equation}
then $P=P_{\mu}$. This correspondence allows one to lift a $\Gamma$-invariant
Markov chain on $X$ to a random walk on $\Gamma$ with step distribution
$\mu$. Recall that $dm_{\Gamma}$ is normalized such that $m_{\Gamma}(K)=1$.
Let $d\nu_{K}(g)=dm_{\Gamma}(g)\mathbf{1}_{\{g\in K\}}$ be the restriction
of $m_{\Gamma}$ to $K$. Then $\nu_{K}$ is the normalized Haar measure
on the compact group $K$. Note that since $\mu$ is bi-$K$-invariant,
$\nu_{K}\ast\mu^{(n)}=\mu^{(n)}$ and it is bi-$K$-invariant. Then
by the formula (\ref{eq:mu-P}), we have that

\begin{equation}
P^{n}(o,x)=\frac{d\mu^{(n)}}{dm_{\Gamma}}(g_{x}),\ \mbox{for any }g_{x}\ \mbox{such that }g_{x}\cdot o=x.\label{eq:P}
\end{equation}
It follows that for entropy we have the identity
\begin{equation}
H_{P}(n)=-\int_{\Gamma}\log\left(\frac{d\mu^{(n)}}{dm_{\Gamma}}(g)\right)d\mu^{(n)}(g).\label{eq:Hn}
\end{equation}

We now focus on $\Gamma$-invariant Markov chains on $X$ that are
symmetric, $P(x,y)=P(y,x)$ for all $x,y\in V$. Consider the associated
Dirichlet form 
\[
\mathcal{E}_{P}(f)=\frac{1}{2}\sum_{x,y\in G}(f(x)-f(y))^{2}P(x,y),
\]
and define 
\begin{equation}
\lambda_{P}(\Omega)=\inf\{\mathcal{E}_{P}(f):\mbox{support}(f)\subset\Omega,\|f\|_{\ell^{2}(X)}=1\}.\label{def-eig-1}
\end{equation}
One can define the spectral profile $\Lambda_{P}$ in the same way
as before. Let $\rho(P)=\sup_{f\neq\mathbf{0}}\frac{\left\Vert Pf\right\Vert _{\ell^{2}(X)}}{\left\Vert f\right\Vert _{\ell^{2}(X)}}$
be the spectral radius of $P$. A fundamental result due to Soardi
and Woess \cite{Soardi1990} states that the spectral radius of an
irreducible, $\Gamma$-invariant and symmetric Markov operator $P$
is $1$ if and only if $\Gamma$ is both amenable and unimodular. 

\begin{proof}[Proof of Theorem \ref{transitive}]

The assumption on the sub-exponential decay of $P^{n}(o,o)$ implies
that $\rho(P)=1$, therefore by \cite{Soardi1990} we have that $\Gamma$
is both amenable and unimodular. 

Note that when $\Gamma$ is unimodular, a symmetric $\Gamma$-invariant
Markov operator $P$ lifts to a symmetric probability measure on $\Gamma$.
Indeed, by (\ref{eq:P}),
\begin{align*}
d\mu(g^{-1}) & =P(o,g^{-1}\cdot o)dm_{\Gamma}(g^{-1})=P(g^{-1}\cdot o,o)dm_{\Gamma}(g^{-1})\ \ \ (\mbox{symmetry})\\
 & =P(o,g\cdot o)dm_{\Gamma}(g^{-1})\ \ \ (\mbox{invariance under }\Gamma)\\
 & =P(o,g\cdot o)dm_{\Gamma}(g)\ \ \ (\mbox{unimodularity})\\
 & =d\mu(g).
\end{align*}

The Coulhon-Grigor'yan theory that relates the spectral profile of
$P$ to the decay of $P^{(2n)}(o,o)$ works for more general reversible
random walks on graphs, see \cite{CNash}. In particular, Lemma \ref{Lambda-upper}
is valid with $\mu^{(2n)}(\id)$ replaced by $P^{2n}(o,o)$ and $\Lambda_{\mu}$
replaced by $\Lambda_{P}$. Note also that the Markov type inequality
(\ref{eq:markov}) holds for general symmetric probability measure
$\mu$ on a locally compact group $G$. 

Given a finite set $U$ in the graph $X$, we can lift $U$ up to
a set $\tilde{U}$ in $\Gamma$, that is
\[
\tilde{U}=\cup\{g_{x}K,\ g_{x}\cdot o\in U\}.
\]
Denote by $\tilde{W}_{n}$ a random walk on $\Gamma$ with step distribution
$\mu$. Lemma \ref{tail} holds for $\tilde{W}_{n}$:
\[
\mathbf{P}\left(W_{n}\notin\tilde{U}^{-1}\tilde{U}\right)\le n\lambda_{\mu}\left(\tilde{U}\right)\le n\lambda_{P}\left(U\right).
\]
It is crucial here that $\mu$ is a symmetric probability measure
on $\Gamma$, which is a consequence of unimodularity of $\Gamma$.
Then as in Section \ref{sec:entropy}, the proof of Theorem \ref{liouville}
and \ref{bound} goes through for $(G,\mu)$ and yields upper estimates
for $H_{\mu}(n)$ as stated. By the identity (\ref{eq:Hn}), the same
conclusions hold for $H_{P}(n)$. 

\end{proof}

Finally, we remark that Theorem \ref{moment} holds for a space homogenous
symmetric Markov operator $P$ with $\lambda_{\mu}\left(B(\id,r)\right)$
replaced by $\lambda_{P}(B(o,r))$, and the distance $\left|\cdot\right|$
taken to be the graph distance on $X$. The proof follows the same
procedure of lifting $P$ to a probability measure on $\Gamma$. Again,
the assumption on the decay of the spectral profile of balls implies
that $\rho(P)=1$, therefore the measure $\mu$ on $\Gamma$ is symmetric.

\medskip{}

\textbf{Acknowledgements.} We thank Micha\l{} Kotowski and Balint
Virág for providing Figure 1, and Ryokichi Tanaka for reading the
manuscript carefully and giving helpful comments. We thank the anonymous
referees for comments that improved the paper.

\bibliographystyle{alpha}
\bibliography{Thesis}

\def\cprime{$'$}
\begin{thebibliography}{dCTV07}

\bibitem[AK16]{amir2016groups}
Gideon Amir and Gady Kozma.
\newblock Groups with minimal harmonic functions as small as you like.
\newblock {\em arXiv preprint arXiv:1605.07593}, 2016.

\bibitem[Ale92]{Alex}
G.~Alexopoulos.
\newblock A lower estimate for central probabilities on polycyclic groups.
\newblock {\em Canad. J. Math.}, 44(5):897--910, 1992.

\bibitem[ANP09]{Austin2009}
Tim Austin, Assaf Naor, and Yuval Peres.
\newblock The wreath product of {$\Bbb Z$} with {$\Bbb Z$} has {H}ilbert
  compression exponent {$\frac{2}{3}$}.
\newblock {\em Proc. Amer. Math. Soc.}, 137(1):85--90, 2009.

\bibitem[AV12]{Amir2012}
Gideon Amir and B{\'a}lint Vir{\'a}g.
\newblock Speed exponents of random walks on groups.
\newblock {\em To appear in {I}nternational {M}athematics {R}esearch {N}otices,
  arXiv:1203.6226}, 2012.

\bibitem[Ave76]{Avez1976}
A.~Avez.
\newblock Harmonic functions on groups.
\newblock In {\em Differential geometry and relativity}, pages 27--32.
  Mathematical Phys. and Appl. Math., Vol. 3. Reidel, Dordrecht, 1976.

\bibitem[BE16]{Bartholdi2016}
Laurent Bartholdi and Anna Erschler.
\newblock Poisson--furstenberg boundary and growth of groups.
\newblock {\em Probability Theory and Related Fields}, pages 1--26, 2016.

\bibitem[BGT87]{BGT}
N.~H. Bingham, C.~M. Goldie, and J.~L. Teugels.
\newblock {\em Regular variation}, volume~27 of {\em Encyclopedia of
  Mathematics and its Applications}.
\newblock Cambridge University Press, Cambridge, 1987.

\bibitem[BHM08]{Blach`ere2008}
S{\'e}bastien Blach{\`e}re, Peter Ha{\"{\i}}ssinsky, and Pierre Mathieu.
\newblock Asymptotic entropy and {G}reen speed for random walks on countable
  groups.
\newblock {\em Ann. Probab.}, 36(3):1134--1152, 2008.

\bibitem[CG97]{CG}
Thierry Coulhon and Alexander Grigor'yan.
\newblock On-diagonal lower bounds for heat kernels and {M}arkov chains.
\newblock {\em Duke Math. J.}, 89(1):133--199, 1997.

\bibitem[CGP01]{CGP}
T.~Coulhon, A.~Grigor'yan, and C.~Pittet.
\newblock A geometric approach to on-diagonal heat kernel lower bounds on
  groups.
\newblock {\em Ann. Inst. Fourier (Grenoble)}, 51(6):1763--1827, 2001.

\bibitem[Cou96]{CNash}
Thierry Coulhon.
\newblock Ultracontractivity and {N}ash type inequalities.
\newblock {\em J. Funct. Anal.}, 141(2):510--539, 1996.

\bibitem[dCTV07]{Cornulier2007}
Yves de~Cornulier, Romain Tessera, and Alain Valette.
\newblock Isometric group actions on {H}ilbert spaces: growth of cocycles.
\newblock {\em Geom. Funct. Anal.}, 17(3):770--792, 2007.

\bibitem[Der80]{Derriennic1980}
Yves Derriennic.
\newblock Quelques applications du th\'eor\`eme ergodique sous-additif.
\newblock 74:183--201, 4, 1980.

\bibitem[Ers03]{Erschlerdrift}
Anna Erschler.
\newblock On drift and entropy growth for random walks on groups.
\newblock {\em Ann. Probab.}, 31(3):1193--1204, 2003.

\bibitem[Ers06]{Erschler2006}
Anna Erschler.
\newblock Piecewise automatic groups.
\newblock {\em Duke Math. J.}, 134(3):591--613, 2006.

\bibitem[GK04]{Guentner2004}
Erik Guentner and Jerome Kaminker.
\newblock Exactness and uniform embeddability of discrete groups.
\newblock {\em J. London Math. Soc. (2)}, 70(3):703--718, 2004.

\bibitem[Gou16]{Gournay2014}
Antoine Gournay.
\newblock The {L}iouville property and {H}ilbertian compression.
\newblock {\em Ann. Inst. Fourier (Grenoble)}, 66(6):2435--2454, 2016.

\bibitem[Kai91]{Kaimanovich1991}
Vadim~A. Kaimanovich.
\newblock Poisson boundaries of random walks on discrete solvable groups.
\newblock In {\em Probability measures on groups, {X} ({O}berwolfach, 1990)},
  pages 205--238. Plenum, New York, 1991.

\bibitem[KV83]{KV}
V.~A. Ka{\u\i}manovich and A.~M. Vershik.
\newblock Random walks on discrete groups: boundary and entropy.
\newblock {\em Ann. Probab.}, 11(3):457--490, 1983.

\bibitem[KV15]{Kotowski2015}
Micha{\l} Kotowski and B{\'a}lint Vir{\'a}g.
\newblock Non-{L}iouville groups with return probability exponent at most 1/2.
\newblock {\em Electron. Commun. Probab.}, 20:no. 12, 2015.

\bibitem[KW02]{Kaimanovich2002}
Vadim~A. Kaimanovich and Wolfgang Woess.
\newblock Boundary and entropy of space homogeneous {M}arkov chains.
\newblock {\em Ann. Probab.}, 30(1):323--363, 2002.

\bibitem[LP13]{Lee2013}
James~R Lee and Yuva Peres.
\newblock Harmonic maps on amenable groups and a diffusive lower bound for
  random walks.
\newblock {\em The Annals of Probability}, 41(5):3392--3419, 2013.

\bibitem[LP16]{LPbook}
Russell Lyons and Yuval Peres.
\newblock {\em Probability on Trees and Networks}.
\newblock Cambridge University Press, 2016.

\bibitem[NP08]{Naor2008}
Assaf Naor and Yuval Peres.
\newblock Embeddings of discrete groups and the speed of random walks.
\newblock {\em Int. Math. Res. Not. IMRN}, 2008.

\bibitem[PSC00]{PSCstab}
Ch. Pittet and L.~Saloff-Coste.
\newblock On the stability of the behavior of random walks on groups.
\newblock {\em J. Geom. Anal.}, 10(4):713--737, 2000.

\bibitem[PSC02]{Pittet2002}
C.~Pittet and L.~Saloff-Coste.
\newblock On random walks on wreath products.
\newblock {\em Ann. Probab.}, 30, no.2:948--977, 2002.

\bibitem[PSC03]{PSCrank}
Ch. Pittet and L.~Saloff-Coste.
\newblock Random walks on finite rank solvable groups.
\newblock {\em J. Eur. Math. Soc. (JEMS)}, 5(4):313--342, 2003.

\bibitem[SCZ15]{perm}
Laurent Saloff-Coste and Tianyi Zheng.
\newblock Isoperimetric profiles and random walks on some permutation wreath
  products.
\newblock {\em To appear in Revista Matematica Iberoamericana,
  arXiv:1510.08830}, 2015.

\bibitem[SCZ16]{Saloff-Coste2014}
Laurent Saloff-Coste and Tianyi Zheng.
\newblock Random walks and isoperimetric profiles under moment conditions.
\newblock {\em Ann. Probab.}, 44(6):4133--4183, 2016.

\bibitem[SW90]{Soardi1990}
Paolo~M. Soardi and Wolfgang Woess.
\newblock Amenability, unimodularity, and the spectral radius of random walks
  on infinite graphs.
\newblock {\em Math. Z.}, 205(3):471--486, 1990.

\bibitem[Tes11]{Tessera2011}
Romain Tessera.
\newblock Asymptotic isoperimetry on groups and uniform embeddings into
  {B}anach spaces.
\newblock {\em Commentarii Mathematici Helvetici}, 86(3):499--535, 2011.

\bibitem[Tes13]{Tessera2013}
Romain Tessera.
\newblock Isoperimetric profile and random walks on locally compact solvable
  groups.
\newblock {\em Rev. Mat. Iberoam.}, 29(2):715--737, 2013.

\end{thebibliography}

\textsc{\newline Yuval Peres \newline One Microsoft Way, Redmond, WA 98052} \newline \textit{E-mail address:} peres@microsoft.com 

\textsc{\newline Tianyi Zheng \newline Department of Mathematics, UC San Diego, 9500 Gilman Dr. La Jolla, CA 92093} \newline \textit{E-mail address:} tzheng2@math.ucsd.edu 
\end{document}